\documentclass[letterpaper,12pt]{amsart}
\usepackage{latexsym}
\usepackage{amsmath,amssymb,amsthm}
\usepackage{amscd}
\usepackage[sorted]{amsrefs}
\usepackage{url}
\usepackage{color}

\title{Reversing Palindromic Enumeration in Rank Two Free Groups}

\author{Andrew E. Silverio}
\address{SentriLock LLC, 45069}
\email{andrewsilverio@gmail.com}
\thanks{The results here are part of the author's Ph.D.~thesis. The author thanks the Rutgers-Newark Mathematics Department for its support during the course of his graduate studies and the Graduate School Newark for a Dissertation Fellowship.}

\newtheorem{theorem}{Theorem}[section]
\newtheorem{corollary}[theorem]{Corollary}

\newtheorem{definition}[theorem]{Definition}
\newtheorem{lemma}[theorem]{Lemma}
\numberwithin{equation}{section}

\begin{document}

\begin{abstract}

The Gilman-Maskit algorithm for determining the discreteness or non-discreteness of a two-generator subgroup of $\mathrm{PSL}_2\mathbb{R}$ terminates with a pair of generators that are Farey words \cite{algorithm}. The Farey words are primitive words that are indexed by rational numbers and infinity. The so-called E-words \cite{enumeration}, primitive words with palindromic or palindromic product forms, are also indexed by rational numbers and infinity. We produce a modification of the Gilman-Maskit algorithm so that the stopping generators are E-words and what can be considered a new palindromic enumeration scheme. The original definition of the enumeration scheme can be implemented and run in a machine without any modification. However every time a recursion \textit{calls itself}, the state of the \textit{previous caller} is stored until the recursion stops calling itself. It is often efficient for a recursion to minimize calling itself in order to avoid wasted resources such as time and storage space. The non-recursive formulas for special cases reduce the self-calling of the recursion.
\end{abstract}
\maketitle
\section{Introduction}
The difficulty of determining the discreteness of a group of M\"obius transformations has lead to the development of algorithms and iterative methods. Some methods or procedures utilize J{\o}rgensen's Inequality \cite{MR0427627} and the Poincar\'e Polyhedron Theorem \cite{MR1393195} which supply necessary conditions for the non-discreteness and discreteness respectively. The J{\o}rgensen number of a pair of matrices $(A,B)$ is $J(A,B)$ where $J(A,B)= \lvert \mathrm{tr}[A,B]-4\rvert + \lvert \mathrm{tr}(A)^2 -2 \rvert$. A non-elementary group is not discrete if there is a pair of elements $A$ and $B$ where $J(A,B)<1$. The Poincar\'e Polygon Theorem concludes that a group is discrete if there is a finite sided polygon and a set of side pairings generating the group which satisfy a number of properties.
Thus applying the Poincar\'e Polygon Theorem requires verifying that a number of hypotheses are satisfied.

\par R.~Riley's procedure \cite{MR689477} to determine the discreteness of a group $G$ starts with a finite set of generators $S_1$ of $G$ and examines all pairs $(A,B)$ $\in S_1 \times S_1$ for the for $J(A,B)<1$ condition. If this is not satisfied, then the procedure tries to construct a candidate fundamental domain and a side-pairing using elements of  $S_1$ to test the hypothesis of the Poincar\'e Polygon Theorem. This test is applied to all subsets of $S_1$.  If a determination of discreteness or non-discreteness is not made, $S_1$ is replaced by $S_2$ which is a set of all words of length at most two in the generators in $S_1$, and the process is repeated. This procedure, called Riley's procedure, does not always terminate at some $S_n$. Riley's method is an example of a \emph{semi-algorithm}, a procedure which is a finite algorithm in some cases but not in others. In any input, the length and number of words the semi-algorithm considers grows exponentially.

\par Gilman and Maskit have a finite algorithm for the case of two-generator subgroups of $\mathrm{PSL}_2\mathbb R$ where each element of any pair of generators considered can be any geometric type (elliptic, parabolic, hyperbolic) as long as any pair of hyperbolic generators have disjoint axes. Gilman has an additional algorithm for the case when any pair of hyperbolic generators has intersecting axes \cite{MR1290281}. The computational complexity of Riley's procedure and the Gilman-Maskit  algorithm is analyzed in \cite{MR2699935} and in \cite{MR1290281}. Riley's procedure is double exponential. Y.~C.~Jiang showed that Gilman-Maskit algorithm is of polynomial time \cite{Jiang}.
\par More generally, Gilman and Keen \cite{GKsemi} have a procedure that is a semi-algorithm for two-generator non-elementary subgroups of $\mathrm{PSL}_2\mathbb C$. Because of the existence of geometrically infinite groups, it is thought that there could be no complete algorithm.

\par In \cite{enumeration} a new iteration scheme for primitive words in rank-two free group is given. It is called the \emph{palindromic enumeration scheme} as each word, called an \emph{E-word}, is either the unique palindromic word in its conjugacy class or is given as the unique element in its conjugacy class that is a product of two palindromes that appeared at the previous step in the enumeration scheme and \cite{GKsemi} used the concept of the algorithm as a two stage automata initiated here.  The words in the enumeration scheme correspond to the rational numbers and thus can be labeled $E_{p/q}$ where $p/q$ is a rational number and its continued fraction expansion is given by $[a_0; a_1, a_2,\ldots, a_k]$.

\par The goal here is to shorten some steps in the Gilman-Maskit algorithm. These should also apply to the Gilman-Keen procedure and might also apply to shorten the Riley procedure when applied to an initial two-generator group. 
Here we present two alternative ways of producing and studying the E-words defined in \cite{enumeration}, first by reducing the number of recursions in the palindromic enumeration scheme and second by defining a non-recursive iterative scheme that comes from modifying the Gilman-Maskit algorithm \cite{algorithm}.

\section{Organization}

\par  In section \ref{section:earlytermination}, alternative but equivalent conditions for terminating the recursive iteration are shown. Particularly, the enumeration scheme can be terminated on two conditions: when the input is either an integer or the reciprocal of an integer. This result helps in implementing the E-enumeration scheme in a program we call the \textit{E-word Calculator}. A sample source code can be found on \url{https://github.com/andrewsilverio/EwordsEnumeration}.

\par In section \ref{section:modified}, the Gilman-Maskit algorithm is modified so that the stopping generators are E-words.


\par Finally, we apply the investigations of the modified algorithm to obtain a theorem, Theorem \ref{theorem:length} about the number of E-words of a given length. Examples are provided in section \ref{section:examples}.

\section{Equivalent Conditions for Terminating Palindromic Primitives} \label{section:earlytermination}

It is well known that the conjugacy classes of the primitive elements of a rank-2 free group $F_2$ can be indexed by the rational numbers and infinity up to taking inverses. Moreover, it is also known that for each conjugacy class of primitive elements, there is a representative that is either a palindrome or product of two palindromes \cite{enumeration}. The palindromic enumeration scheme enumerates all primitive words by defining a function $E: \mathbb{Q} \cup \{\infty\} \to F_2 = \langle A, B \rangle$. This function is recursive and terminates on conditions $0 \mapsto A^{-1}$ and $\infty \mapsto B$. In this section, we provide non-recursive formulas for this function $E$ in cases where the rational number is an integer or reciprocal of an integer. These formulas serve as an alternative set of two terminating conditions applied to the palindromic enumerating scheme derived by Gilman and Keen \cite{enumeration}.

\par The non-recursive formulas for special cases reduce the self-calling of the recursion in the definition of the palindromic enumeration scheme.

\subsection{Summary of the Palindromic Enumeration Scheme}

The notation used here for elements of $\mathbb{Q} \cup \{ \infty \}$ are of the form $p/q$ where $ p \in \mathbb{Z}$, $ q \in \mathbb{Z} \cap [ 0 , \infty)$ and $\mathrm{gcd}(p,q)=1$. The element $\infty$ is denoted by $1/0$. By definition, $\frac{1}{0}$ and $\frac{0}{1}$ are in lowest terms.
\begin{definition}
Let $p/q, r/s \in \mathbb{Q} \cup \{ \infty\}$. The pair $p/q$ and $r/s$ are called \emph{Farey neighbors} if $\lvert ps - rq \rvert =1$.
\par If $p/q$ and $r/s$ are Farey neighbors, the \emph{Farey sum} of $p/q$ and $r/s$ is \[ \frac{p}{q} \oplus \frac{r}{s} = \frac{ p+r}{q+s}.\]
\end{definition}
Both $p/q$ and $r/s$ are Farey neighbors of their Farey sum. The Farey neighbors do not have the transitive property. A non-integer rational number may have infinitely many Farey neighbors but the set of its Farey neighbors is certainly bounded. We provide a name for its minimum and maximum such neighbor.
\begin{definition}
The smallest and largest Farey neighbors of a nonzero rational number $p/q$ are called \emph{parents} of $p/q$.
\end{definition}

The following is the definition of the palindromic enumeration scheme found in \cite{enumeration}.
 Set $E_{0/1}=A^{-1}$ and $E_{1/0}=B$. For the rest of $\mathbb{Q}$, take the parents $m/n$ and $r/s$ of $p/q$ such that $\frac{m}{n} < \frac{p}{q} < \frac{r}{s}$. Define $E_{p/q}$ recursively by
\[E_{p/q}= \begin{cases}
E_{r/s}E_{m/n} & \text{if $pq$ is odd,}\\
E_{m/n}E_{r/s} & \text{if $pq$ is even.}
\end{cases}\]
\begin{definition}
The function $E: \mathbb{Q} \cup \{ \infty \} \to F_2$ given by $p/q \mapsto E_{p/q}$ is called the \emph{palindromic enumeration scheme} or simply \emph{enumeration scheme}. \par The definitions of $E_{0/1}$ and $E_{1/0}$ are called \emph{terminal conditions} since they do not require breaking a fraction into the Farey sum of their parents. Hence, we call the elements $0$ and $\infty$ of $\mathbb{Q}\cup \{ \infty \}$ \emph{orphans}.
\end{definition}

\subsection{Non-recursive Formulas for Special Cases}
In this section, formulas are given for $E_{n/1}$ and $E_{1/n}$ for all $n \in \mathbb{Z}$. Since the enumeration scheme is a recursive definition, the corresponding words of non-orphans are cumbersome to compute. However, formulas can be derived in some cases. The following are facts about parents of integers.
\begin{lemma} \label{lemma:integershortcut}
For $n> 1$, the parents of $\frac{1}{n}$ are $\frac{0}{1}$ and $\frac{1}{n-1}$. For $n>0$, the parents of $\frac{n}{1}$ are $ \frac{1}{0}$ and $ \frac{n-1}{1}$.
\end{lemma}
\begin{proof}
Since $\frac{1}{0}= \infty$, any other Farey neighbor of $\frac{n}{1}$ must be finite. Suppose $p/q$ is a finite Farey neighbor of $n$. We may assume $q\geq 1$; otherwise, pass the negative sign to $p$. Then, $\frac{p}{q} < n \Rightarrow p < qn$. Since $p/q$ is a Farey neighbor of $n$, $\lvert p-qn \rvert =1$. Hence, $qn-p=1$, and
\begin{eqnarray*}
 q\geq qn-p & \Longrightarrow & p \geq qn -q \\
 \nonumber & \Longrightarrow & p\geq q(n-1)\\
 \nonumber & \Longrightarrow & \frac{p}{q} \geq n-1.
\end{eqnarray*}
Since $n-1$ is a Farey neighbor of $n$, $n-1$ must be the lower parent of $n$.
\par Next, we find the parents of $\frac{1}{n}$. Suppose $n > 1$ and $\frac{p}{q}$ is a Farey neighbor of $\frac{1}{n}$ with $\frac{p}{q} > \frac{1}{n}$. Since $n> 1$, $\frac{1}{n} > 0$ so we assume $p \geq 1$ and $q \geq 1$. Then $\lvert pn-q \rvert =1$ and $pn>q \Rightarrow pn -q =1$. Hence,
\begin{eqnarray*}
 p \geq 1& \Longrightarrow & p \geq pn -q \\
 \nonumber & \Longrightarrow & q \geq pn -p\\
 \nonumber & \Longrightarrow & q \geq p(n-1) \\
 \nonumber & \Longrightarrow & \frac{1}{n-1} \geq \frac{p}{q}.
\end{eqnarray*}
Since $\frac{1}{n-1}$ is a Farey neighbor of $\frac{1}{n}$, it is the greater parent of $\frac{1}{n}$.
Lastly, if a Farey neighbor $\frac{p}{q} \leq \frac{1}{n}$, then $q-pn=1 \Rightarrow q-1=pn \Rightarrow p=\frac{q-1}{n}$. Since $n>1$ and we may assume that $q \geq 1$, it implies $p\geq 0$. Hence $\frac{p}{q} \geq 0$. Since 0 is a Farey neighbor of $\frac{1}{n}$, $\frac{0}{1}$ must be the lower parent of $\frac{1}{n}$.
\end{proof}

\begin{corollary}
If $n$ is a negative integer, the parents of $n$ are $\infty$ and $n+1$; the parents of $\frac{1}{n}$ are $0$ and $\frac{1}{n+1}$ for $n<-1$.
\end{corollary}
\begin{proof}
If $n<0$, then $n+1$ is the greatest Farey neighbor of $n$ other than $\infty$. Using similar methods, $\infty$ is the lowest possible parent of a negative rational number. On the other hand, if $ n < -1$, then $ -n >1$, so the parents of $\frac{1}{-n}$ are 0 and $\frac{1}{-n-1}=-\frac{1}{n+1}$. Hence, the minimum and maximum Farey neighbors of $\frac{1}{n}$ are $\frac{-1}{-(n+1)}$ and 0 respectively. 
\end{proof}

In computing a primitive word in the image of the enumeration scheme, the recursion eventually runs through the decreasing entries of a continued fraction $[a_0;a_1, \ldots, a_k]$. In particular, the parents of $\left[ a_0;a_1,a_2,\ldots, a_k\right]$ are the fractions with continued fraction expansions $\left[a_0;a_1,a_2,\ldots, a_k -1\right]$ and $\left[ a_0; a_1,a_2,\ldots, a_{k-1}\right]$ \cites{enumeration}. These parents are broken down into Farey sums of their corresponding parents, and eventually the recursion encounters $[a_0;]$ or $[0;a_1]$. The fraction $\frac{n}{1}$ has the form $[n;]$ and $\frac{1}{n}$ has the form $[0;n]$. Thus, the formulas for $\frac n1$ and $\frac 1n$ save the iteration several steps. To construct more unified formulas, a function $s: \mathbb{R} \to \{-1,1\}$ is defined by \[ s(x) =\begin{cases}
1 & \text{for $x \in (-\infty,0)$},\\
-1 & \text{for $x \in [0, \infty)$.}
\end{cases}\]
\begin{theorem} \label{theorem:shortcut}
Let $n \in \mathbb{Z}$. Then,
\[E_\frac{n}{1} = B^{\left\lceil\frac{\lvert n\rvert}{2}\right\rceil}A^{s(n)}B^{\left\lfloor\frac{\lvert n\rvert}{2}\right\rfloor}\]
and
\[E_\frac{1}{n}= A^{s(n)\left\lfloor\frac{\lvert n\rvert}{2}\right\rfloor}BA^{s(n)\left\lceil\frac{\lvert n\rvert}{2}\right\rceil}.\]
\end{theorem}

\par The formula works for $n=-2,-1,0,1$ and $2$. The rest of the integers can be verified using inductive steps $n+2$ and $n-2$.

%

\subsection{Alternative Termination Conditions} \label{section:shortcut}
Since using Theorem \ref{theorem:shortcut} allows the enumeration scheme to terminate the recursion earlier, we conclude this section with the alternative but equivalent terminating conditions.
\begin{theorem}
The palindromic enumeration scheme can have its recursion terminated using the conditions
\[E_\frac{n}{1} = B^{\left\lceil\frac{\lvert n\rvert}{2}\right\rceil}A^{-1}B^{\left\lfloor\frac{\lvert n\rvert}{2}\right\rfloor} \qquad
E_\frac{1}{n}= A^{-\left\lfloor\frac{\lvert n\rvert}{2}\right\rfloor}BA^{-\left\lceil\frac{\lvert n\rvert}{2}\right\rceil}\] for $n \in \mathbb{Z}\cap [0,\infty)$ ; and
\[E_\frac{n}{1} = B^{\left\lceil\frac{\lvert n\rvert}{2}\right\rceil}AB^{\left\lfloor\frac{\lvert n\rvert}{2}\right\rfloor} \qquad
E_\frac{1}{n}= A^{\left\lfloor\frac{\lvert n\rvert}{2}\right\rfloor}BA^{\left\lceil\frac{\lvert n\rvert}{2}\right\rceil}\] for $n \in \mathbb{Z}\cap (-\infty,0)$.
\end{theorem}
\begin{proof}
If $p/q=[a_0;a_1,\ldots,a_k]$, then the parents of $p/q$ are $[a_0;a_1,\ldots,a_{k-1}]$ and $[a_0;a_1,\ldots,a_k-1]$ from \cite{enumeration}. The splitting process of the enumeration scheme eventually queries the E-word corresponding to either $[a_0;]$ or $[0;a_1]$. Both of these correspond to $E_{a_0}$ and $E_{1/a_0}$ respectively.
\end{proof}

\section{The Modified Gilman-Maskit Algorithm} \label{section:modified}
\par The main idea of a \textit{step} is to replace one of the two generators with their product, and the new idea is to view the procedure as a two stage automata.
A \textit{linear step} in the Gilman-Maskit algorithm sends the ordered pair $(g,h)$ to $(g,gh)$. A \textit{Fibonacci step} sends the pair $(g,h)$ to $(gh,g)$ \cites{vidur,MR2581839,algorithm}. Which step is used or picked depends on the traces of the new generators; we may assume that the starting representative matrices have positive traces, that the traces stay positive until the final step, and the one with lower trace occupies the left spot. 
\par By keeping one of the generators, this procedure ensures that the groups generated by the old and new pairs are the same. The algorithm retains the generator with lower trace. The following is the proposed new \textit{step} in picking new generators from a given ordered pair $(a, b)$.

\begin{center}
\begin{tabular}{|l|c|c|}
\hline
conditions for $a$ and $b$ & preserve $a$ & preserve $b$ \\
\hline
both $a$ and $b$ are palindromes & $(a,ba)$ & $(ba,b)$ \\
$a$ is not a palindrome &$(a,ab)$ & $(ab,b)$ \\
$b$ is not a palindrome & $(a,ab)$ & $(ab,b)$ \\
\hline
\end{tabular}
\end{center}
Note that there are no assumptions about the traces of $a$ and $b$, but it assumes $a$ takes the left spot and both generators are either a palindrome or a product of palindromes.

\subsection{Summary of Gilman-Maskit Algorithm}
\par The Gilman-Maskit algorithm takes two elements $A$ and $B$ of $\mathrm{PSL}_2\mathbb{R}$ and gives a definite output: either $\langle A, B \rangle$ is discrete; or not. The algorithm uses conditions, the Poincar\'e polygon theorem or J{\o}rgensen's inequality, to decide whether the group is discrete or not using the generators $A$ and $B$. If it cannot decide using $A$ and $B$, the generators are combined to construct new generators to use for testing discreteness.
\par One such combination is the pair $(A,AB)$ and the traces of their matrices are reduced after the iteration. Eventually the process of changing the generators stops and the algorithm makes a decision \cite{algorithm}.
\par Other combinations and conditions are also used, but the step that changes $(A,B)$ into $(A,AB)$, called here a Nielsen step, is the main modification of this section. Note the step here termed a Nieslen step is one of the many types of Nielsen moves on a pair of generators. 

\subsection{New Linear and Fibonacci Steps}
The original linear step preserves the left generator and changes the other. The original Fibonacci step turns the left generator into the right generator, and hence changes both generators. The F-sequence in \cite{wordsequence} records the consecutive linear steps before a Fibonacci step or the algorithm stops. Thus, it defines an ordered set of positive integers $(n_0,n_1,\ldots,n_k)$ where each $n_i$ corresponds to the number of consecutive linear steps.
\par The proposed new steps here always preserve one generator including its position whether left or right. Instead of classifying the steps, we define a new sequence $[n_0;n_1,n_2,\ldots,n_k]$ called an \emph{E-sequence}. Let $n_0$ be the number of steps that preserve the initial right generator before changing it. Let $n_1$ be the number of steps in preserving the initial left generator before changing it. Let $n_2$ be the number of steps the next right generator is preserved. The rest of the $n_i$'s alternate between left and right generators. So for even $i$, $n_i$ steps preserve the right generator; for odd $i$, $n_i$ steps preserve the left generator. If the algorithm preserves the left generator first, we let $n_0=0$. All $n_i$'s assume positive integer values except $n_0$ which can take a zero value. Hence, an E-sequence can take continued fraction expansion values of any positive rational number $p/q$. 
\subsection{Reversing the Enumeration Scheme}
In this section, reversing the process of the palindromic enumeration scheme is shown. The definition of the enumeration scheme requires taking the parents of a given rational number. While the parents exist and are well-defined for most rational numbers, their computations and ordering are cumbersome. In addition, the parents are broken further into \textit{grandparents} until orphans are encountered. Every time a parent is not an orphan, another splitting into two parents must occur; the manual computations get worse.
\par Instead of starting with a fraction and breaking it into the Farey sum of its parents, one can start with the greatest grandparents of all other elements which exactly are the orphans $0$ and $\infty$. This section explains in detail how this process can be done. The pair of parents of a typical fraction are Farey neighbors, and thus share the properties of Farey neighbors outlined below. Furthermore, their Farey sum is equal to their \textit{only child}. In the process of reversing the enumeration scheme, the modification of the Gilman-Maskit algorithm, we also prove  that the process stops with E-words. 
\par The following are facts about Farey neighbors. 
\begin{lemma} \label{lemma:first}
Let $\frac{p}{q}$ and $\frac rs$ be Farey neighbors with $\frac{p}{q} < \frac{r}{s}$. Then $\frac{p}{q}<\frac{p+r}{q+s}<\frac{r}{s}$; and the pairs $\frac{p}{q},\frac{p+r}{q+s}$ and $\frac{p+r}{q+s},\frac{r}{s}$ are Farey neighbors.
\end{lemma}

\begin{lemma}
Let $\frac{p}{q}$ and $\frac{r}{s}$ be Farey neighbors. Then $p$ and $q$ cannot be both even; $r$ and $s$ cannot be both even. Moreover, $p$, $q$, $r$ and $s$ cannot be all odd.
\end{lemma}
\begin{proof}
If $p$ and $q$ are both even, then $\frac{p}{q}$ are not in lowest terms since $\mathrm{gcd}(p,q) \geq 2$. Same is true with $r$ and $s$. Suppose all integers $p$, $q$, $r$ and $s$ are odd. Then $ps$ and $rq$ are also odd, but $ps-qr$ is even. In particular $\lvert ps -rq \rvert $ is not $1$.
\end{proof}
\begin{lemma}
For each pair of Farey neighbors $\frac{p}{q}$ and $\frac{r}{s}$, only one of the following combinations hold.
\begin{enumerate}
\item $pq$ is odd; $rs$ is even.
\item $pq$ and $rs$ are even.
\item $pq$ is even $rs$ is odd.
\end{enumerate}
\end{lemma}
\begin{proof}
If $pq$ is odd, both $p$ and $q$ are odd. By the lemma above, $r$ and $s$ cannot be both odd so one of them must be even. Hence $rs$ is even.
\end{proof}
When it comes to listing possibilities of the integers $p$, $q$, $r$ and $s$ whether even or odd, two more combinations can be eliminated.
\begin{lemma}
For each pair of Farey neighbors $\frac{p}{q}$ and $\frac{r}{s}$, the following combinations do not hold.
\begin{enumerate}
\item $p$ and $r$ are even; $q$ and $s$ are odd.
\item $p$ and $r$ are odd; $q$ and $s$ are even.
\end{enumerate}
\end{lemma}
\begin{proof}
If the fractions are Farey neighbors, $\lvert ps-rq \rvert =1$. If any combinations above hold, then both $ps$ and $rq$ are even. Hence, the difference of even numbers is even. In particular $\lvert ps-rq \rvert$ cannot equal to $1$.
\end{proof}
Initially, odd-even combinations of four integers $p$, $q$, $r$ and $s$ add up to $16$. However the preceding lemmas imply that there can be only $6$ possibilities.
\begin{theorem} \label{theorem:onlyone}
For each pair of Farey neighbors $\frac{p}{q}$ and $\frac{r}{s}$, only one of the following combinations hold.
\end{theorem}
\begin{center}
\begin{tabular}{|cccc|c|}
\hline
$p$& $q$& $r$ & $s$ & $(p+r)(q+s)$ \\
\hline
even & odd & odd & even & odd\\
odd & odd & even & odd & even\\
odd & odd & odd & even & even\\
odd & even & even & odd & odd\\
even & odd & odd & odd & even\\
odd & even & odd & odd & even\\
\hline
\end{tabular}
\end{center}
\begin{proof}
The conditions where $p$ and $q$ are both even eliminate four conditions. The conditions where $r$ and $s$ are both even reduces $3$ more. The case where all $p$, $q$, $r$ and $s$ are odd and lemma above make $3$ more impossible. The remaining possibilities are $6$ out of $16$.
\end{proof}
The palindromic enumeration scheme maps each element of $\mathbb{Q}\cup \{ \infty \}$ to a primitive word in $F_2$ by defining a recursion. Recall that the word $E_{p/q}$ in $F_2$ corresponding to a positive rational number $p/q$ is a word in $E_{m/n}$ and $E_{r/s}$.
The elements $\frac{m}{n}$ and $\frac{r}{s}$ of $\mathbb{Q}\cup \{ \infty \}$ are the parents of $\frac{p}{q}$. For the orphans $0$ and $\infty$, $E_0=A^{-1}$ and $E_\infty=B$. It is also assumed that $\frac{m}{n} < \frac{r}{s} \leq \infty$.
\par It is well-known that the Farey sum of the parents of $p/q$ is equal to $p/q$. It is also known that for any Farey neighbors $\frac{m}{n}$ and $\frac{r}{s}$ whose Farey sum is $\frac{p}{q}$, then $\frac{m}{n}$ and $\frac{r}{s}$ must be the parents of $\frac{p}{q}$. 
\begin{lemma}
If $\frac{p}{q}$ and $\frac{r}{s}$ are positive and Farey neighbors, then the parents of their Farey sum are exactly $\frac{p}{q}$ and $\frac{r}{s}$.
\end{lemma}
\par The lemma above and Lemma \ref{lemma:first} allow one to guess the parents of a given rational number $p/q$. This is done by breaking $p$ and $q$ into sums $p=m+r$ and $q=n+s$, so that $\lvert ms -rn \rvert =1$. On large numerators or denominators, the combinations of sums can be cumbersome, but the goal is to reverse the process of the recursion defined in the enumeration scheme. More precisely, the goal is to determine $E_{p/q}$ starting from $A$ and $B$ instead of starting from computing the parents of $p/q$.

\par Henceforth, we assume Farey neighbors $\frac{p}{q}$ and $\frac{r}{s}$ have $\frac{p}{q}<\frac{r}{s}$. The following theorem derives the E-word corresponding to $\frac{p+r}{q+s}$.
\begin{theorem}
Let $\frac{p}{q}$ and $\frac{r}{s}$ be nonnegative and Farey neighbors with $\frac{p}{q}<\frac{r}{s}$. Then, $E_{(p+r)/(q+s)}$ is a product of $E_{p/q}$ and $E_{r/s}$ determined by the following table.
\end{theorem}
\begin{center}
\begin{tabular}{|cccc|c|c|}
\hline
$p$& $q$& $r$ & $s$ & $(p+r)(q+s)$ & $E_{(p+r)/(q+s)}$ \\
\hline
even & odd & odd & even & odd & $E_{r/s}E_{p/q}$\\
odd & odd & even & odd & even & $E_{p/q}E_{r/s}$\\
odd & odd & odd & even & even & $E_{p/q}E_{r/s}$\\
odd & even & even & odd & odd & $E_{r/s}E_{p/q}$\\
even & odd & odd & odd & even & $E_{p/q}E_{r/s}$\\
odd & even & odd & odd & even & $E_{p/q}E_{r/s}$\\
\hline
\end{tabular}
\end{center}
\begin{proof}
Since the parents of $\frac{p+r}{q+s}$ are exactly $\frac{p}{q}$ and $\frac{r}{s}$, the residue class mod 2 of $(p+r)(q+s)$ can be determined by the residue class mod 2 of $p$, $q$, $r$, and $s$. The possible combinations are fully listed. In any case, the E-word corresponding to $\frac{p+r}{q+s}$ is determined in terms of the words corresponding to $E_{p/q}$ and $E_{r/s}$.
\end{proof}
\par The image of the enumeration scheme is a set of palindromes or product of palindromes. Gilman and Keen \cite{enumeration} proved that $E_{p/q}$ is a palindrome if and only if $pq$ is even. Hence, $E_{p/q}$ is not a palindrome if and only if $pq$ is odd. Using the table in the theorem above, $E_{(p+r)/(q+s)}$ is a palindrome if either $pq$ or $rs$ is odd; and $E_{(p+r)/(q+s)}$ is not a palindrome if both $pq$ and $rs$ are even.
\par Let $a=A^{-1}$ and $b=B$, where $A$ and $B$ are generators of rank-2 free group. Then $\langle a, b \rangle =\langle A,B \rangle $ and $(a,b) = \left( E_{0/1},E_{1/0} \right )$. This initial pair has the rational number corresponding to the left generator less than that of the right generator.
\begin{theorem}\label{theorem:withtable}
Let $\frac{p}{q} < \frac{r}{s}$ be positive Farey neighbors. Let $\left ( a_1, b_1 \right)$ be the new pair of generators after applying a step of the modified algorithm to the generators $\left( E_{p/q},E_{r/s} \right )$. Then both $a_1$ and $b_1$ are E-words; $a_1= E_{j/k}$ and $b_1= E_{m/n}$ such that $\frac{j}{k} < \frac{m}{n}$. Either $\frac{j}{k}$ or $\frac{m}{n}$ is equal to $\frac{p+r}{q+s}$ so $\frac{j}{k}$ and $\frac{m}{n}$ are Farey neighbors.
\end{theorem}
\begin{proof}
Let $a_0=E_{p/q}$ and $b_0=E_{r/s}$. Then $\left( a_1,b_1 \right)$ is one of the following.
\begin{center}
\begin{tabular}{|l|c|c|}
\hline
conditions for $a_0$ and $b_0$ & preserve $a_0$ & preserve $b_0$ \\
\hline
both $a_0$ and $b_0$ are palindromes & $(a_0,b_0a_0)$ & $(b_0a_0,b_0)$ \\
$a_0$ is not a palindrome &$(a_0,a_0b_0)$ & $(a_0b_0,b_0)$ \\
$b_0$ is not a palindrome & $(a_0,a_0b_0)$ & $(a_0b_0,b_0)$ \\
\hline
\end{tabular}
\end{center}
Since either $a_0$ or $b_0$ is preserved, we show that $a_0b_0$ or $b_0a_0$ is the E-word corresponding to $\frac{p+r}{q+s}$. If $a_0$ and $b_0$ are both palindromes, then both $pq$ and $rs$ are even. Using the table in Theorem \ref{theorem:onlyone}, $(p+r)(q+s)$ is odd in any possible combinations of parities (residue classes mod 2) of $p$, $q$, $r$ and $s$. Also $\frac{p}{q}$ and $\frac{r}{s}$ are the parents of $\frac{p+r}{q+s}$. Hence, $E_{(p+r)/(q+s)}= E_{r/s}E_{p/q}=b_0a_0$.
\par If either $a_0$ or $b_0$ is not a palindrome, then either $pq$ or $rs$ is odd, respectively. The same table shows $(p+r)(q+s)$ is even so $E_{(p+r)/(q+s)}= E_{r/s}E_{p/q}=a_0b_0$.
\par The only thing left to show is that $\frac{j}{k} < \frac{m}{n}$. This is an application of Lemma \ref{lemma:first} stating that $\frac{p}{q} < \frac{p+r}{q+s} < \frac{r}{s}$. If $\frac{j}{k} = \frac{p}{q}$, then $\frac{m}{n} =\frac{p+r}{q+s}$ so $\frac{j}{k}$ and $\frac{m}{n}$ are Farey neighbors. If $\frac{j}{k} =\frac{p+r}{q+s}$, then $\frac{m}{n}=\frac{r}{s}$. In any case $\frac{j}{k}< \frac{m}{n}$.
\end{proof}
\begin{corollary}
The modified algorithm steps, applied finitely many times to a pair of primitive associates $\left( E_{p/q}, E_{r/s} \right)$ where $\frac{p}{q} < \frac{r}{s}$, stop with a pair of E-words that generate the same group $\left \langle E_{p/q}, E_{r/s} \right \rangle = \langle A,B \rangle$. 
\end{corollary}

\subsection{Consecutive Steps}
In the theory of F-sequences, taking $n$ consecutive linear steps is given by the simple formula $(a,b) \mapsto (a, a^nb)$. We note that $a$ is preserved in each of the $n$ consecutive steps. In the modified algorithm, there is more than one formula, and not all of them are as simple. The formulas depend on the palindromic conditions of the current generators and upon which generator is preserved. There are six formulas shown in the following.

\begin{center}
\begin{tabular}{|l|c|c|}
\hline
\mbox{} & preserve $a$ & preserve $b$ \\
\hline
both $a$ and $b$ are palindromes & $\displaystyle \left( a, E_{\frac{1}{n}}(a,b) \right)$ & $\displaystyle \bigl( E_{n}(a,b),b \bigr)$\\
$a$ is not a palindrome & $\displaystyle\left( a, a^nb \right)$ & $\displaystyle \left( E_{\frac{1}{n}}(b,a), b \right)$\\
$b$ is not a palindrome & $\displaystyle \bigl( a, E_{n}(b,a) \bigr)$ & $\displaystyle\left( ab^n, b \right)$\\
\hline
\end{tabular}
\end{center}

\subsection{From E-sequences to E-words}
The main purpose of the modification is to end the algorithm with E-words. Since $\mathrm{tr}^2(ab) = \mathrm{tr}^2(ba)$ and the modification uses only Nielsen automorphisms, the proposed method stops the algorithm with the same number of steps and complexity as the original one. In this section, we prove that this modification produces E-words in the end. More precisely and more strongly,

\begin{theorem} \label{theorem:esequence} Let $[n_0;n_1,n_2,\ldots,n_k]$ be the continued fraction expansion of the nonnegative rational number $p/q$. Then the last changed generator of the modified Gilman-Maskit algorithm using the E-sequence $[n_0;n_1,n_2,\ldots,n_k]$ is the E-word corresponding the rational number $-p/q$.
\end{theorem}
\begin{proof}

\par Let $[n_0; n_1, n_2, \ldots , n_k]$ be an E-sequence. Then the modified algorithm has outputs of E-words corresponding to the rational numbers $p_i/q_i$ and $r_i/s_i$ for $i=0,1,2,\ldots,k$ in the following recursive formulas.
\begin{align*}
p_0 &= 0 & q_0 &= 1 & r_0 &= 1 & s_0 &=0
\end{align*}
\begin{align*}
p_i&= n_{2i-2}r_{i-1} +p_{i-1} & r_i &= n_{2i-1}p_i +r_{i-1}\\
q_i&= n_{2i-2}s_{i-1} +q_{i-1} & s_i &= n_{2i-1}q_i +s_{i-1}
\end{align*}
The next thing to show is that
\[ \frac{p_i}{q_i} = \left [ n_0; n_1, n_2, \ldots , n_{2i-2} \right] \] and
\[ \frac{r_i}{s_i} = \left [ n_0; n_1, n_2, \ldots , n_{2i-1} \right]. \]
\par There are formulas in \cite{enumeration} where \textit{approximants} are defined as follows.
\begin{align*}
g_0&= n_0 & h_0 &=1 & g_1&=n_1n_0+1 & h_1&= n_1
\end{align*}
\begin{align*}
g_i &= n_ig_{i-1}+g_{i-2}\\
h_i &= n_ih_{i-1}+h_{i-2}
\end{align*}
It was claimed in \cite{enumeration} that $\displaystyle \frac{g_i}{h_i} = \left[ n_0 ; n_1, n_2, \ldots, n_i \right]$. There is a way to relate $\frac{p_i}{q_i}$ and $\frac{r_i}{s_i}$ to $\frac{g_i}{h_i}$. In particular,
\begin{align*}
\frac{p_i}{q_i} &= \frac{g_{2i-2}}{h_{2i-2}} & \frac{r_i}{s_i} &= \frac{g_{2i-1}}{h_{2i-1}}
\end{align*}
for all $i \geq 1$. We show it as follows.
\begin{align*}
p_1&= n_0r_0+p_0& r_1&= n_1p_1+r_0 \\
\nonumber &=n_0& \nonumber &= n_1n_0+1\\
\nonumber &=g_0 & \nonumber &= g_1 \\
q_1&= n_0s_0 +q_0 & s_1&= n_1q_1+s_0\\
\nonumber &=n_0 \cdot 0 +1& \nonumber &=n_1\cdot 1 + 0 \\
\nonumber &= 1 & \nonumber &=n_1\\
\nonumber &= h_0 & \nonumber &=h_1\\
\end{align*}
The assertions work for $i=1$. To show that the formulas work for all other $i$, we show that they work for $i+1$. That is,
\begin{align*}
p_{i+1}&= g_{2(i+1) -2}& q_{i+1}&= h_{2i}\\
\nonumber &= g_{2i+2-2}& r_{i+1}&= g_{2i+1}\\
\nonumber &= g_{2i}& s_{i+1}&= h_{2i+1}.
\end{align*}
The following are the computations.
\begin{align*}
p_{i+1}&= n_{2i}r_i +p_i & r_{i+1}&=n_{2i+1}p_{i+1}+r_i\\
\nonumber &=n_{2i}g_{2i-1}+g_{2i-2}& \nonumber &= n_{2i+1}g_{2i}+g_{2i-1}\\
\nonumber &=g_{2i}& \nonumber &=g_{2i+1}\\
q_{i+1}&= n_{2i}s_i +q_i& s_{i+1}&= n_{2i+1}q_{i+1}+s_i\\
\nonumber &= n_{2i}h_{2i-1}+h_{2i-2}& \nonumber &= n_{2i+1}h_{2i}+h_{2i-1}\\
\nonumber &= h_{2i}& \nonumber &= h_{2i+1}
\end{align*}
\par Now, except for $p_0$, $q_0$, $r_0$, and $s_0$, all $p_i$, $q_i$, $r_i$, and $s_i$ are consolidated into the formulas for $g_i$ and $h_i$ so that
\[ \frac{g_i}{h_i} = \left[ n_0 ; n_1, n_2, \ldots, n_i \right]. \]
\end{proof}

\subsection{Forms of E-words}
It is shown in section \ref{section:shortcut} what the form of E-words is if $k=0$ in the E-sequence or precisely $p/q= [n_0;] \in \mathbb{Z}$. In this section, we look at E-words up to $k=2$ and show what happens to the \textit{exponents} of $a$ and $b$ for any E-word. To alleviate complicated notations, we define $m_i =\lfloor \frac{n_i}{2} \rfloor$ and $M_i =\lceil\frac{n_i}{2} \rceil$. Then $m_i +M_i =n_i$ and $m_i=M_i$ if $n_i$ is even; $M_i=m_i+1$ if $n_i$ is odd. This notation is used as superscripts. For example, Theorem \ref{theorem:shortcut} shows $E_{n_0} = b^{M_i}ab^{m_i}$ and $E_{1/n_1} = a^{m_i}ba^{M_i}$.

\par The following theorem shows the length of a given E-word $E_{m/n}$.
\begin{theorem}
The length of $E_{m/n}$ in the generator set $\left\{a,b\right\}$ is $\displaystyle \lvert m \rvert + \lvert n \rvert$. Moreover $E_{m/n}$ has $\lvert m \rvert$ $b$-factors and $\lvert n \rvert$ $a$-factors.
\end{theorem}
\begin{proof}
Suppose $g, h \in F_2$; $g$ has $p$ $b$-factors and $q$ $a$-factors. Suppose $h$ has $r$ $b$-factors and $s$ $a$-factors. The modified algorithm applied to $(g,h)$ replaces one of the generators with either $gh$ or $hg$. Both $gh$ and $hg$ have $p+r$ $b$-factors and $q+s$ $a$-factors. The algorithm starts with $(a,b)$ and ends with $\left( E_{p/q}, E_{r/s}\right)$. The fraction $1/1$ correspond to $ba$ which is the very first new E-word of the algorithm. For $E_{1/1}=ba$, the assertion is true. Suppose this assertion is still true after the algorithm stops at the pair $\left( E_{p/q}, E_{r/s}\right)$. Then $E_{p/q}$ has $p$ $b$-factors and $q$ $a$-factors; $E_{r/s}$ has $r$ $b$-factors and $s$ $a$-factors. Continuing the algorithm just one step further yields a new E-word $E_{(p+r)/(q+s)}$. It is either $E_{p/q}E_{r/s}$ or $E_{r/s}E_{p/q}$. In any case, $E_{(p+r)/(q+s)}$ has $p+r$ $b$-factors and $q+s$ $a$-factors.
\end{proof}

\par The next theorem shows how many E-words there are of length $n$.

\begin{theorem} \label{theorem:length}
Let $n \in \mathbb{N}$. There exists a bijection between the set of nonzero integers relatively prime with $n$ in the interval $[-n,n] \subset \mathbb{R}$ and the set of E-words of length $n$.
\end{theorem}
\begin{proof}
Let $\Phi$ and $\Psi$ be the sets defined as follows.
\begin{align*}
 \Phi &=\left\{ p \in \mathbb{Z} : 0 < \lvert p \rvert < n \text{ and } \mathrm{gcd}(p,n) =1 \right \}\\
 \Psi &= \left\{ (p,q) \in \mathbb{Z}\times\mathbb{N} : \lvert p \rvert + q = n \text{, } p \neq 0 \text{ and } \mathrm{gcd}(p,q) = 1 \right\}
\end{align*}
\par Define a function $f: \Phi \to \Psi$ given by $f(x) = \left( x, n- \lvert x \rvert \right)$. Let $x\in \Phi$. Then $\lvert x \rvert + \left( n- \lvert x \rvert \right)=n$, $\mathrm{gcd}(x,n)=1$ and $n-\lvert x\rvert > 0$. Also $\mathrm{gcd}\bigl(x, n - \lvert x \rvert \bigr)=1$ since a divisor of both $x$ and $n- \lvert x \rvert$ is also a divisor of $n$. Hence $f(x) \in \Psi$. Let $(p,q) \in \Psi$. Then $q=n-\lvert p \rvert$ and $\mathrm{gcd}(p,n)=1$ by a similar argument. Thus, $p\in \Phi$ and $f(p)= (p,q)$. If $x \neq y \in \Phi$, $f(x) = \left( x, n-\lvert x\rvert \right) \neq \left( y, n- \lvert y \rvert \right) = f(y)$. The set $\Psi$ is exactly the index set of all E-words of length $n$.
\end{proof}

\par\indent For $p/q > 1$ and $k \geq 2$, four possible cases of an E-sequence can be fed into the modified Gilman-Maskit algorithm. The E-word corresponding to $\left [ n_0; n_1, \ldots, n_{k} \right]$ is also a word in $E \left([ n_0; n_1, \ldots, n_{k-2}] \right )$ and $E \left([ n_0; n_1, \ldots, n_{k-1}] \right )$ if $k\geq 2$. By induction, $E \left([ n_0; n_1, \ldots, n_{k}] \right )$ is a word in $E \left([ n_0; n_1, n_{2}] \right )$ and $E \left([ n_0; n_1] \right )$.

\par \mbox{}

\par\noindent Summary of cases $k \leq 2$:
\begin{center}
\begin{tabular}{|l|c|}
\hline
E-sequence & stopping pair \\
\hline
$\displaystyle\left[ \textrm{odd}; n_1 \right]$ & $\displaystyle \left(b^{M_0}ab^{m_0}, b^{M_0}\left(ab^{n_0}\right)^{n_1-1}ab^{M_0} \right)$\\
$\displaystyle\left[ \textrm{even}; n_1 \right]$ & $\displaystyle \left(b^{m_0}ab^{m_0}, b^{m_0}\left(ab^{n_0} \right)^{m_1-1}ab^{n_0+1}\left(ab^{n_0}\right)^{M_1-1}ab^{m_0} \right)$\\
$\displaystyle\left[ \textrm{odd}; 1, n_2 \right]$ & $\displaystyle \left(b^{M_0}\left(ab^{n_0+1}\right)^{m_2}ab^{n_0}\left(ab^{n_0+1}\right)^{M_2-1}ab^{M_0},b^{M_0}ab^{M_0} \right)$\\
$\displaystyle\left[ \textrm{even}; 1, n_2 \right]$ & $\displaystyle \Bigl(b^{m_0}\left(ab^{n_0+1} \right)^{n_2}ab^{m_0}, b^{m_0+1}ab^{m_0} \Bigr)$\\
\hline
\end{tabular}
\end{center}
\par Note that the formulas above work even if $n_k=1$.
\begin{corollary}\label{kleq}
For $n_0>0$, $k\geq 3$ and $p/q=\left[ n_0; n_1,\ldots, n_k \right]$, $ E_{p/q}$ is a word in $a$ and $b$ of the form \[ b^{k_1}ab^{k_2}ab^{k_3}\cdots ab^{k_q}ab^{k_{q+1}}\] where $k_1, k_{q+1} \in \left\{m_0,M_0\right\}$ and $\left\{k_2,k_3,\ldots, k_q \right\} = \left\{n_0,n_0+1 \right\}$.
\end{corollary}
\begin{proof}
By Theorem \ref{theorem:esequence}, $E_{p/q}$ is the last modified generator in running the E-sequence $\left[n_0;n_1,\ldots,n_k \right]$. The modified algorithm has to output the words $E\left(\left[n_0;n_1\right]\right)$ and $E\left(\left[n_0;n_1,n_2\right]\right)$ in the middle of the process. Hence $E_{p/q}$ is a word in $E\left(\left[n_0;n_1\right]\right)$ and $E\left(\left[n_0;n_1,n_2\right]\right)$. The possibilities of these words are listed in the preceding table. Any product or powers of them has $b^{m_0}b^{m_0}$, $b^{m_0}b^{m_0+1}$, $b^{M_0}b^{M_0}$, or $b^{m_0}b^{M_0}$ in its substring which simplifies to either $b^{n_0}$ or $b^{n_0+1}$. The table also shows that $E_{p/q}$ is of the claimed form for $i=2,\ldots,q$, and $k_1$ and $k_{q+1}$ are in $\left\{M_0, m_0\right\}$.

\par\indent In the second case in the table, $m_1$ and $M_1$ can possibly equal to $1$ so there might be no $b$-exponent equal to $n_0$. We show that $n_0$ still appears as an exponent in $E\left(\left[ n_0;n_1,\ldots,n_k\right]\right)$ if $k\geq 3 $.
\par\indent Let $n_1>1$. Suppose after running the E-sequence $[n_0;n_1]$ on $(a,b)$, the new pair is $(g,h)$. If $n_0$ is even, then $g$ is a palindrome. Continuing the E-sequence further to $[n_0;n_1,1]$ turns $g$ into either $gh$ or $hg$.
\begin{align*}
gh &= b^{m_0}ab^{n_0}\left(ab^{n_0}\right)^{m_1-1}ab^{n_0+1}\left(ab^{n_0}\right)^{M_1-1}ab^{m_0}\\
\nonumber &= b^{m_0}\left(ab^{n_0}\right)^{m_1}ab^{n_0+1}\left(ab^{n_0}\right)^{M_1-1}ab^{m_0}\\
hg &= b^{m_0}\left(ab^{n_0}\right)^{m_1-1}ab^{n_0+1}\left(ab^{n_0}\right)^{M_1-1}ab^{n_0}ab^{m_0}\\
\nonumber &= b^{m_0}\left(ab^{n_0}\right)^{m_1-1}ab^{n_0+1}\left(ab^{n_0}\right)^{M_1}ab^{m_0}
\end{align*}
Hence $E\left(\left[n_0;n_1,n_2\right]\right)$, for $n_2>1$, is a word in $E\left(\left[n_0;n_1\right]\right)$ and $E\left(\left[n_0;n_1,1\right]\right)$, but $E\left(\left[n_0;n_1,n_2,n_3\right]\right)$ is a word in $E\left(\left[n_0;n_1\right]\right)$ and $E\left(\left[n_0;n_1,n_2\right]\right)$. Thus, there is an $i$ with $2\leq i \leq q$ so that $k_i=n_0$.
\par The same phenomenon occurs in the last case where both primitive generators do not have $n_0$ as a $b$-exponent. We must also show that $n_0$ still appears as an exponent in $E\left(\left[n_0;n_1,\ldots,n_k\right]\right)$ for $k \geq 3$.
\par Let $n_2\geq 1$. Suppose after running the E-sequence $[n_0;1,n_2]$ on $(a,b)$, the new pair is $(g,h)$. If $n_0$ is even, $g$ is a palindrome and $h$ is not. Then the E-sequence $\left[n_0;1,n_2,1\right]$ turns $(g,h)$ into $(g,gh)$. Continuing further, $\left[n_0;1,n_2,2\right]$ stops with the pair $\left(g,ghg\right)$ while $\left[n_0;1,n_2,1,1\right]$ stops with the pair $\left( ghg,gh\right)$. Thus $E\left(\left[n_0;1,n_2,1,n_4\right]\right)$ is a word in $gh$ and $ghg$ for $n_4>1$, and $E\left(\left[n_0;1,n_2,n_3\right]\right)$ is a word in $g$ and $ghg$ for $n_3>2$.
\begin{align*}
ghg &= b^{m_0}\left(ab^{n_0+1}\right)^{n_2}ab^{m_0}b^{m_0+1}ab^{m_0}b^{m_0}\left(ab^{n_0+1}\right)^{n_2}ab^{m_0}\\
\nonumber &= b^{m_0}\left(ab^{n_0+1}\right)^{n_2}ab^{n_0+1}ab^{n_0}\left(ab^{n_0+1}\right)^{n_2}ab^{m_0}\\
gh &= b^{m_0}\left(ab^{n_0+1}\right)^{n_2}ab^{m_0}b^{m_0+1}ab^{m_0}\\
\nonumber &= b^{m_0}\left(ab^{n_0+1}\right)^{n_2}ab^{n_0+1}ab^{m_0}
\end{align*}
Hence both $E\left(\left[n_0;1,n_2,1,n_4\right]\right)$ and $E\left(\left[n_0;1,n_2,n_3\right]\right)$ have $b$-exponents equal to $n_0$.
\par \emph{Remark:} The E-word corresponding to $\left[n_0;1,n_2,1\right]$ still has no $b$-exponent equal to $n_0$. 
However, we make a convention, to avoid ambiguity in the continued fraction expansion $\left[n_0;n_1,\ldots,n_k\right]$. It is the convention that the last entry $n_k$ must be at least $2$.

\end{proof}

\par The case when $p/q<1$ is similar but $n_0=0$ and the index shifts by $1$. For example, the E-sequence $\left[ 0; n_1,n_2,n_3 \right]$ is similar to $\left[ n_1; n_2, n_3 \right]$ which is greater than $1$ as a rational number. 

\mbox{}

\par\noindent Summary of cases for $n_0=0$ and $k \leq 3$:
\begin{center}
\begin{tabular}{|l|c|}
\hline
E-sequence & stopping pair \\
\hline
$\displaystyle\left[0; \textrm{odd}; n_2 \right]$ & $\displaystyle \bigl( a^{M_1}\left(ba^{n_1}\right)^{n_2-1}ba^{M_1},a^{m_1}ba^{M_1} \bigr)$\\
$\displaystyle\left[0; \textrm{even}; n_2 \right]$ & $\displaystyle \left(a^{m_1}\left(ba^{n_1}\right)^{M_2-1}ba^{n_1+1}\left(ba^{n_1}\right)^{m_2-1}ba^{m_1}, a^{m_1}ba^{m_1} \right)$\\
$\displaystyle\left[0; \textrm{odd}; 1, n_3 \right]$ & $\displaystyle \left( a^{M_1}ba^{M_1}, a^{M_1}\left(ba^{n_1+1}\right)^{M_3-1}ba^{n_1}\left(ba^{n_1+1}\right)^{m_3}ba^{M_1} \right) $\\
$\displaystyle\left[0; \textrm{even}; 1, n_3 \right]$ & $\displaystyle \Bigl( a^{m_1}ba^{m_1+1}, a^{m_1}\left(ba^{n_1+1}\right)^{n_3}ba^{m_1} \Bigr)$\\
\hline
\end{tabular}
\end{center}
\par Likewise, the table works for $n_2$ or $n_3$ possibly equal to $1$.
\begin{corollary}
For $n_0=0$, $n_1 \geq1$, $k\geq 4$ and $p/q =\left[ n_0;n_1,\ldots,n_k\right]$, $E_{p/q}$ is a word in $a$ and $b$ of the form \[a^{k_1}ba^{k_2}ba^{k_3}\cdots ba^{k_p}ba^{k_{p+1}}\] where $k_1, k_{p+1} \in \left\{ m_1,M_1 \right\}$ and $\left\{k_2,k_3,\ldots,k_p\right\}=\left\{n_1,n_1+1\right\}$.
\end{corollary}
\begin{proof}
The E-word corresponding to $\left[0;n_1,n_2,\ldots,n_k\right]$ is a word in the stopping pair of the E-sequence $\left[0;n_1,n_2,n_3\right]$. By induction, $E_{\left[0;n_1,\ldots,n_k\right]}$ is a word in the stopping pair of $\left[0;n_1,n_2\right]$. From the table above, any word in a given stopping pair has either $a^{M_1+M_1}$, $a^{m_1+m_1}$, $a^{m_1+1+m_1}$, or $a^{m_1+M_1}$ in its substring which is equal to either $a^{n_1}$ or $a^{n_1+1}$. Hence, $E_{\left[0;n_1,\cdots,n_k\right]}$ is of the form $a^{k_1}ba^{k_2}b\cdots a^{k_p}ba^{k_{p+1}}$ where $k_1, k_{p+1} \in \left\{m_1,M_1\right\}$ and $k_2,k_3,\ldots,k_p \in \left\{ n_1,n_1+1\right\}$. Using similar arguments in Corollary \ref{kleq}, there is an $i$ for which $k_i =n_1$ if $k \geq 4$.
\end{proof}



\section{Appendix} \label{section:examples}
The following shows the modified algorithm using the E-sequence $[5;4,3]$.
\begin{align*}
\left(a,b\right)& \to \left(ba,b\right) \to \left(bab,b\right) \to \left(b^2ab,b\right) \to \left(b^2ab^2,b\right) \\
\nonumber & \to \left(b^3ab^2,b\right) \to \left(b^3ab^2,b^3ab^3\right) \\
\nonumber & \to \left(b^3ab^2,b^3ab^5ab^3\right) \to \left(b^3ab^2,b^3ab^5ab^5ab^3\right) \\
\nonumber & \to \left(b^3ab^2,b^3ab^5ab^5ab^5ab^3\right) \to \left(b^3ab^5ab^5ab^5ab^3,b^3ab^5ab^5ab^5ab^3\right) \\
\nonumber & \to \left(b^3ab^5ab^5ab^5ab^6ab^5ab^5ab^5ab^5ab^3,b^3ab^5ab^5ab^5ab^3\right) \\
\nonumber & \to \left(b^3ab^5ab^5ab^5ab^6ab^5ab^5ab^5ab^5ab^6ab^5ab^5ab^5ab^3,b^3ab^5ab^5ab^5ab^3\right)
\end{align*}
One can observe that the sum of the exponents of $a$ is $13$, and that of $b$ is $68$. Moreover, the continued fraction expansion of $\frac{68}{13}$ is $[5;4,3]$.\\

The following shows the algorithm using the E-sequence $[4;3,2]$.
\begin{align*}
\left(a,b\right)& \to \left(ba,b\right) \to \cdots \to \left(b^2ab^2,b\right) \to \left(b^2ab^2,b^3ab^2\right) \\
\nonumber & \to \left(b^2ab^2,b^2ab^5ab^2\right) \to \left(b^2ab^2, b^2ab^5ab^4ab^2\right) \\
\nonumber & \to \left(b^2ab^4ab^5ab^4ab^2,b^2ab^5ab^4ab^2\right) \\
\nonumber & \to \left(b^2ab^4ab^5ab^4ab^4ab^5ab^4ab^2, b^2ab^5ab^4ab^2\right)
\end{align*}
In this example, the sum of the exponents of $a$ is $7$, and that of $b$ is $30$. Likewise, the continued fraction expansion of $\frac{30}{7}$ is $[4;3,2]$.

The following shows the algorithm using the E-sequence $[0;3,4]$.
\begin{align*}
\left(a,b\right)& \to \left(a,ba\right) \to \left(a,aba\right) \to \left(a,aba^2\right) \to \left(a^2ba^2,aba^2\right) \\
\nonumber & \to \left(a^2ba^3ba^2,aba^2\right) \to \left(a^2ba^3ba^3ba^2,aba^2\right) \\
\nonumber & \to \left(a^2ba^3ba^3ba^3ba^2,aba^2\right)
\end{align*}
The fraction for above is $\frac{4}{13}$ which has a continued fraction expansion of $[0;3,4]$.

Lastly, the example found in \cite{enumeration} is copied to the E-sequence $[3;2,4]$.
\begin{align*}
\left(a,b\right)& \to \left(ba,b\right) \to \left(bab,b\right) \to \left(b^2ab,b\right) \to \left(b^2ab,b^2ab^2\right)\\
\nonumber & \to \left(b^2ab,b^2ab^3ab^2\right) \to \left(b^2ab^3ab^3ab^2,b^2ab^3ab^2\right) \\
\nonumber & \to \left(b^2ab^3ab^4ab^3ab^3ab^2,b^2ab^3ab^2\right) \\
\nonumber & \to \left(b^2ab^3ab^4ab^3ab^3ab^4ab^3ab^2,b^2ab^3ab^2\right) \\
\nonumber & \to \left(b^2ab^3ab^4ab^3ab^4ab^3ab^3ab^4ab^3ab^2,b^2ab^3ab^2\right)
\end{align*}
The last modified generator looks like the E-word corresponding to $\frac{31}{9}$.

\section{Acknowledgement} The author thanks his thesis adviser, Professor Jane Gilman, for her patience and guidance, Professor Linda Keen for some useful conversations, and the referee for innumerable valuable suggestions and corrections.


\begin{bibdiv}
\begin{biblist}

\bib{MR1393195}{book}{
   author={Beardon, Alan F.},
   title={The geometry of discrete groups},
   series={Graduate Texts in Mathematics},
   volume={91},
   note={Corrected reprint of the 1983 original},
   publisher={Springer-Verlag, New York},
   date={1995},
   pages={xii+337},
   isbn={0-387-90788-2},
   review={\MR{1393195 (97d:22011)}},
}

\bib{algorithm}{article}{
  author={Gilman, J.},
  author={Maskit, B.},
  title={An algorithm for $2$-generator Fuchsian groups},
  journal={Michigan Math. J.},
  volume={38},
  date={1991},
  number={1},
  pages={13--32},
  issn={0026-2285},
  review={\MR{1091506 (92f:30062)}},
  doi={10.1307/mmj/1029004258},
}

\bib{MR1290281}{article}{
  author={Gilman, Jane},
  title={Two-generator discrete subgroups of ${\rm PSL}(2,{\bf R})$},
  journal={Mem. Amer. Math. Soc.},
  volume={117},
  date={1995},
  number={561},
  pages={x+204},
  issn={0065-9266},
  review={\MR{1290281 (97a:20082)}},
}

\bib{discon}{article}{
  author={Gilman, Jane},
  title={A discreteness condition for subgroups of ${\rm PSL}(2,{\bf C})$},
  conference={
   title={Lipa's legacy},
   address={New York},
   date={1995},
  },
  book={
   series={Contemp. Math.},
   volume={211},
   publisher={Amer. Math. Soc.},
   place={Providence, RI},
  },
  date={1997},
  pages={261--267},
  review={\MR{1476991 (98k:30060)}},
  doi={10.1090/conm/211/02824},
}

\bib{GKsemi}{article}{
	author={Gilman, Jane},
	author={Keen, Linda},
	title={Canonical hexagons and the $\mathrm{PSL}(2, \mathbb C)$ discreteness problem},
	note={preprint},
}
\bib{wordsequence}{article}{
  author={Gilman, Jane},
  author={Keen, Linda},
  title={Word sequences and intersection numbers},
  conference={
   title={Complex manifolds and hyperbolic geometry},
   address={Guanajuato},
   date={2001},
  },
  book={
   series={Contemp. Math.},
   volume={311},
   publisher={Amer. Math. Soc.},
   place={Providence, RI},
  },
  date={2002},
  pages={231--249},
  review={\MR{1940172 (2004c:30071)}},
  doi={10.1090/conm/311/05455},
}

\bib{criteria}{article}{
  author={Gilman, Jane},
  author={Keen, Linda},
  title={Discreteness criteria and the hyperbolic geometry of palindromes},
  journal={Conform. Geom. Dyn.},
  volume={13},
  date={2009},
  pages={76--90},
  issn={1088-4173},
  review={\MR{2476657 (2010c:30060)}},
  doi={10.1090/S1088-4173-09-00191-X},
}

\bib{enumeration}{article}{
  author={Gilman, Jane},
  author={Keen, Linda},
  title={Enumerating palindromes and primitives in rank two free groups},
  journal={J. Algebra},
  volume={332},
  date={2011},
  pages={1--13},
  issn={0021-8693},
  review={\MR{2774675 (2012f:20078)}},
  doi={10.1016/j.jalgebra.2011.02.010},
}

\bib{MR2699935}{book}{
   author={Jiang, Yicheng},
   title={Complexity of the Fuchsian group discreteness algorithm},
   note={Thesis (Ph.D.)--Rutgers The State University of New Jersey -
   Newark},
   publisher={ProQuest LLC, Ann Arbor, MI},
   date={2000},
   pages={50},
   isbn={978-0599-51414-0},
   review={\MR{2699935}},
}

\bib{Jiang}{article}{
   author={Jiang, Yicheng},
   title={Polynomial complexity of the Gilman-Maskit discreteness algorithm},
   journal={Ann. Acad. Sci. Fenn. Math.},
   volume={26},
   date={2001},
   number={2},
   pages={375--390},
   issn={1239-629X},
   review={\MR{1833246 (2002d:30053)}},
}





\bib{MR0427627}{article}{
   author={J{\o}rgensen, Troels},
   title={On discrete groups of M\"obius transformations},
   journal={Amer. J. Math.},
   volume={98},
   date={1976},
   number={3},
   pages={739--749},
   issn={0002-9327},
   review={\MR{0427627 (55 \#658)}},
}

\bib{pleating}{article}{
  author={Keen, Linda},
  author={Series, Caroline},
  title={Pleating coordinates for the Maskit embedding of the Teichm\"uller
  space of punctured tori},
  journal={Topology},
  volume={32},
  date={1993},
  number={4},
  pages={719--749},
  issn={0040-9383},
  review={\MR{1241870 (95g:32030)}},
  doi={10.1016/0040-9383(93)90048-Z},
}

\bib{rileyslice}{article}{
  author={Keen, Linda},
  author={Series, Caroline},
  title={The Riley slice of Schottky space},
  journal={Proc. London Math. Soc. (3)},
  volume={69},
  date={1994},
  number={1},
  pages={72--90},
  issn={0024-6115},
  review={\MR{1272421 (95j:32033)}},
  doi={10.1112/plms/s3-69.1.72},
}

\bib{vidur}{book}{
  author={Malik, Vidur},
  title={Curves generated on surfaces by the Gilman-Maskit algorithm},
  note={Thesis (Ph.D.)--Rutgers The State University of New Jersey -
  Newark},
  publisher={ProQuest LLC, Ann Arbor, MI},
  date={2007},
  pages={67},
  isbn={978-0549-21571-4},
  review={\MR{2710734}},
}

\bib{MR2581839}{article}{
  author={Malik, Vidur},
  title={Primitive words and self-intersections of curves on surfaces
  generated by the Gilman-Maskit discreteness algorithm},
  conference={
   title={In the tradition of Ahlfors-Bers. V},
  },
  book={
   series={Contemp. Math.},
   volume={510},
   publisher={Amer. Math. Soc.},
   place={Providence, RI},
  },
  date={2010},
  pages={209--237},
  review={\MR{2581839 (2011e:20072)}},
  doi={10.1090/conm/510/10027},
}



\bib{MR689477}{article}{
   author={Riley, Robert},
   title={Applications of a computer implementation of Poincar\'e's theorem on fundamental polyhedra},
   journal={Math. Comp.},
   volume={40},
   date={1983},
   number={162},
   pages={607--632},
   issn={0025-5718},
   review={\MR{689477 (85b:20064)}},
   doi={10.2307/2007537},
}




\end{biblist}
\end{bibdiv}
\end{document}